\documentclass{amsart}
\usepackage{amsmath}
\usepackage{amsthm}
\usepackage{amssymb}
\usepackage{amscd}
\allowdisplaybreaks[1] 
\newtheorem{thm}{Theorem}
\newtheorem{prop}{Proposition}
\newtheorem{cor}{Corollary}
\newtheorem{lem}{Lemma}
\theoremstyle{definition}
\newtheorem{defi}{Definition}

\renewcommand{\sl}{{\mathfrak{sl}}}
\renewcommand{\d}{\partial}
\newcommand{\om}{\omega}
\newcommand\ot{\otimes}
\newcommand{\g}{\mathfrak g}
\newcommand{\gsl}{\mathfrak{sl}}
\newcommand{\gp}{\mathfrak p} 
 
\newcommand{\gll}{\mathfrak{l}} 
\newcommand{\gl}{\mathfrak{gl}} 
\newcommand{\gk}{\mathfrak k} 
\newcommand{\gs}{\mathfrak s}
\newcommand{\gm}{\mathfrak m}
\newcommand{\gr}{\operatorname{gr}}
\newcommand{\im}{\operatorname{im}}
\newcommand{\Hom}{\operatorname{Hom}}
\newcommand{\Sym}{\operatorname{Sym}}
\newcommand\we{\wedge}

\newcommand{\oo}{\theta}

\newcommand{\G}{\mathcal G}
\newcommand{\CN}{\mathcal{C}N}
\newcommand{\R}{\mathbb R}

\newcommand{\wi}{\tilde}
\newcommand{\tr}{\operatorname{tr}}
\newcommand{\gh}{\mathfrak h}
\newcommand\al{\alpha}
\newcommand{\glm}{{\mathfrak {gl}}_m}

\newcommand{\E}{\mathcal E}
\newcommand{\A}{\mathcal A}
\newcommand{\B}{\mathcal B}
\newcommand{\C}{\mathcal C}
\newcommand{\PP}{\mathcal P}
 
\renewcommand{\d}{\partial}
\newcommand{\Ad}{\operatorname{Ad}}
\newcommand{\ww}{\omega}

\newcommand{\dd}{\operatorname{d}}
\newcommand{\Om}{\Omega}
\newcommand{\so}{\mathfrak {so}}
\newcommand{\som}{{\mathfrak {so}}_{m+2,1}}
\newcommand{\OO}{\Omega}

\theoremstyle{remark}

\title{On the geometry of conformal geodesics equations}
\author{Alexandr Medvedev }
\address{
 Department of Mathematics and Statistics \\ Masaryk  University \\  2, Kotl\'a\v rsk\'a str., 611 37 Brno, Czech Republic}
\thanks{Supported by the project CZ.1.07/2.3.00/20.0003
  of the Operational Programme Education for Competitiveness of the Ministry
  of Education, Youth and Sports of the Czech Republic. The author would like to thank Boris Doubrov, Andreas \v Cap and Vojt\v ech Z\' adn\' ik for helpful conversations.}
\begin{document}

\begin{abstract}

Many geometries can be studied with the help of  distinguished curves. The most known type of such curves are geodesics of Riemannian or projective geometry. In these two cases distinguished curves are specified by a system of the second order ODEs. This fact gives a relation between Riemannian or projective geometries and the geometry of differential equations. 

In the article the conformal geometry is considered. We  answer to the question whether a system of the 3rd order ODEs describes geodesics of a conformal structure. We construct a functor from a category of conformal geometries to a category of Cartan geometries associated to the 3rd order ODEs systems. Explicit formulas which define the family of all equations on conformal geodesics are given in the last section of the article.


\end{abstract}

\maketitle

\section{Introduction}

Distinguished curves of parabolic geometries usually contains a lot of information about underlying geometry.
The geodesics of a projective or Riemannian manifold,  conformal geodesics of a conformal manifold and  Chern-Moser chains of a CR-manifold is a well-known examples of distinguished curves.

 We work with the conformal geometry, which is closely related to the geometry of the third order differential equations. We answer the question whether the given system of the 3rd order ODEs describes conformal geodesics of a conformal structure.

\begin{defi}
A \emph{conformal geodesic} on a conformal manifold $ M $ is a curve on $ M$,
which development in a model space is a circle.
\end{defi}

As it was shown by Yano \cite{yano}, conformal geodesics infinitesimally determines the conformal geometry itself or, more precisely, each infinitesimal symmetry of the defining equation for conformal geodesics is a conformal symmetry. 

Every conformal geodesic is uniquely defined by its second jet. The family of all conformal geodesics of a conformal manifold is defined by a system of third order ODEs \cite{Schouten,Eastwood}.


In the given article we obtain a direct inclusion of the conformal geometry to the geometry of the third order ODEs systems. Both these geometries can be described in terms of Cartan connections. While conformal geometry is parabolic, the geometry associated with the system of the 3rd order ODEs has a model with non-semisimple Lie algebra. 

The geometry of the system of the third order ODEs was studied in \cite{medv2}. The main result of this paper is that to every system of the third order ODEs we can associate an unique characteristic Cartan connection of special type. This fact allows to find fundamental invariants of the system, which are the essential parts of the associated Cartan connection curvature. 

Below in the introduction we remind the basics of the geometry associated with the system of the third order ODEs and describe correspondence space and extension functor construction briefly. Using these two components we obtain an algebraic relation between conformal geometry and the third order ODEs systems geometry in the second part of this paper. In the third part we state and prove necessary and sufficient conditions on conformal geodesics equations. In the final part we list and prove explicit formulas for these conditions.
\subsection{The geometry of the 3rd order ODEs systems}
Consider an arbitrary system of $m$ ordinary differential equations of the third order:
 \begin{equation} \label{1}
 y_i'''(x)=f_i(y_j''(x),y_k'(x),y_l(x),x),
\end{equation}
where $ i,j,k,l=1,\dots,m$ and $m\ge2 .$

Let $J^3(\R^{m+1},1)$ be the third jet space of unparameterized curves. Then the equations~\eqref{1} is a submanifold $\E$ in $J^3(\R^{m+1},1)$. We introduce the following coordinate system on the surface $\E$:
\[(x,y_1,\dots,y_m,p_1=y_1',\dots,p_m=y_m',
q_1=y_1'',\dots,q_m=y_m''). \]

With every system of the third order ODEs we can associate a pair of distributions.  Let $\pi^2_1$ be the canonical projection from the surface $\E$ to the first jet space
$J^1(\R^{m+1},1)$.  We denote a kernel of a differential $\dd\pi^2_1$ as $V$. One-dimensional distribution $E$ is a distribution whose integral curves are the lifts of solutions of equations \eqref{1}. In coordinates distributions $E$, $V$ have the form:
\begin{align*}
E&=\left\langle \frac{\partial}{\partial x}+p_i\frac{\partial}{\partial y_i}+ q_i\frac{\partial}{\partial p_i} + f^i\frac{\partial}{\partial q_i}\right\rangle,\\
V&=\left\langle \frac{\partial}{\partial q_i}\right\rangle,
\end{align*}
where $i,j=1,\dots,m$.

Define a distribution $C$ as the direct sum of the distributions $E$ and $V$. Then $C$ and its subsequent brackets define a filtration of a tangent bundle $T\E$:
  \[C=C^{-1}\subset C^{-2} \subset C^{-3} = T\E ,\]
where $C^{-i-1}=C^{-i}+[C^{-i},C^{-1}]$.

Using the theory developed by Tanaka \cite{tan70} we can effectively construct a Cartan geometry associated to the distribution $C$. This Cartan geometry has an infinitesimal model $(\g,\gh).$ The Lie algebra $\g$ has the following form:
\[
\g=\left(\gsl _2\times \gl_m\right) \rightthreetimes \left(V_2 \otimes W\right),
\]
where $V_2$ is a 3-dimensional irreducible $\sl_2$-module and $W$ is an $m$-dimensional vector space.

Let us fix a basis of the Lie algebra $\gsl_2$ and the $\gsl_2$-module $V_2$. We consider elements $x,y,h$ to be
standard basis of an algebra $\gsl_2$ with relations:
\[
  [x,y]=h,\quad [h,x]=2x, \quad [h,y]=-2y.
\]
Elements ${v_0,v_1,v_2}$ form a basis of the module $V_2$ such that \[x. v_2=v_1,\quad x.v_1=v_0,\quad x.v_0=0.\]
We define the grading of the Lie algebra $\g$ as follows:
 \begin{align*}
\g_1&=\langle y \rangle,&  \g_0&=\langle h,\gl_m \rangle,& &\\ \g_{-1}&=\langle
x
\rangle
+\langle v_2\otimes W \rangle,&  \g_{-2}&=\langle v_1\otimes W \rangle, &
 \g_{-3}&=\langle v_0\otimes W \rangle.  
\end{align*} 
Subalgebra $\gh$ is exactly the nonnegative part of the Lie algebra $\g.$

\begin{defi}
We say that a coframe $\{\omega_{-3}^i,\omega_{-2}^i,\omega_{-1}^i,\omega_x \}$ on $\E$ is \emph{ adapted to the equation} \eqref{1} if:

\begin{itemize}
 \item an annihilator of the forms $\omega_{-3}^i,\omega_{-2}^i,\omega_x$ is $V$;
 \item an annihilator of the forms $\omega_{-3}^i,\omega_{-2}^i,\omega_{-1}^i$ is $E$;
 \item an annihilator of the forms $\omega_{-3}^i$ is $C^{-2}$. 
\end{itemize} 
\end{defi}
Let $\overline{\pi}\colon P\rightarrow\E$ be a principle $H$-bundle and let $\overline\omega$ be and arbitrary Cartan connection of type $(G,H)$ on $P.$ The connection $\overline \omega$ can be written as:
\[
 \overline{\omega }=\overline{\omega }^i_{-3} v_0\otimes e_i +
 \overline{\omega }^i_{-2} v_1\otimes e_i +
 \overline{\omega }^i_{-1} v_2\otimes e_i + \overline{\omega }_x x + \overline{\omega }_h h +
\overline{\omega }^i_j e^j_i + \overline{\omega }_y y ,
\]
where $e_i$ is a basis of vector space $W$ and $e^j_i$ is compatible basis of $\gl_m.$ We also use notation $e^i_{-3},e^i_{-2},e^i_{-1}$ for $v_0\ot e_i,v_1\ot e_i,v_2\ot e_i$ in the future.
\begin{defi}
 We say that a Cartan
connection $\overline{\omega}$ on the principal $H$-bundle $\overline{\pi}$ is \emph{adapted to the equations \eqref{1}}, if
for any local section $s$ of $\overline{\pi}$ the set
 \[\{ s^* \overline{\omega}_x , s^*\overline{\omega}_{-1}^i, s^*\overline{\omega}_{-2}^i,
s^*\overline{\omega}_{-3}^i \}\] is an adapted co-frame on $\E$.
\end{defi}
In the \cite{medv2} it was shown that with every system of the third order ODEs we can uniquely associate a characteristic Cartan connection. This connection is an adapted connection with some restriction on the curvature function. The curvature function can be expressed in terms of fundamental differential invariant with the help of universal covariant derivative. Denote trace-free part of the tensor as $\tr_0.$ There are four fundamental invariants:
\begin{align*} 
\left(W_2\right)^i_j =& \tr_0 \left(\frac{\d f^i}{\d p^j} -
\frac{d}{dx} \frac{\d f^i}{\d q^j} +
\frac{1}{3} \frac{\d f^i}{\d q^k}\frac{\d f^k}{\d q^j} \right), \\
\left(I_2\right)^i_{jk} =& \tr_0\left(\frac{\d^2 f^i}{\d q^j \d
q^k}
 \right), \\
\left(W_3\right)^i_j =&
\frac{\partial f^i}{\partial y^j} +\frac{1}{3} \frac{\partial f^i}{\partial q^k}\frac{\partial f^k}{\partial p^j} - \frac{d}{dx} \frac{\partial f^i}{\partial
p^j} + \frac{2}{3}\frac{d^2}{dx^2} \frac{\partial f^i}{\partial
q^j} + \frac{2}{27}(\frac{\partial f^i}{\partial q^j})^3 \\ 
&-\frac{4}{9}
\frac{\partial f^i}{\partial q^k}\frac{d}{dx}\frac{\partial f^k}{\partial q^j} 
 -
\frac{2}{9} \frac{d}{dx}\left(\frac{\partial f^i}{\partial q^k}\right)\frac{\partial f^k}{\partial q^j} - 2\delta^i_j\frac{d}{dx}H^x ,\\
 \left(I_4\right)_{jk} =& - \frac{\d H_k^{-1}}{\d p_j}+\frac{\d}{\d q_j}\frac{\d}{\d q_k}H^x - \frac{\d}{\d q_k}\frac{d}{dx}H_j^{-1} - \frac{\d }{\d q^k}(H_l^{-1}\frac{\d
f^l}{\d q^j})+2H_j^{-1} H_k^{-1},
\end{align*}  
where \[H_j^{-1}=\frac{1}{6(m+1)}\left(\frac{\d^2 f^i}{\d q^i \d
q^j}
 \right),H^x=-\frac{1}{4m}\left(\frac{\d f^i}{\d p^i} -
\frac{d}{dx} \frac{\d f^i}{\d q^i} +
\frac{1}{3} \frac{\d f^i}{\d q^k}\frac{\d f^k}{\d q^i} \right). \] 

Here and below in the article we use different from \cite{medv2} but equivalent (in terms of fundamental systems of invariants) formula for the invariant $W_3$. The invariants $W_2$ and $W_3$ are called Wilczynski invariants and they appear from linearization of differential equations. The invariants $I_2$ and $I_4$ play crucial role in detecting conformal geodesics equations.
\subsection{Correspondence space}
Correspondence space construction comes from the twistor theory. The main idea is derived from the Penrose correspondence in its general form. Consider a semisimple Lie group $G$ with two parabolic subgroups $P_1$ and  $P_2$. Assume, that  $P_1\cap P_2$ is also parabolic. Then a natural double fibration from  $G/P_1 \cap  P_2$ to $G/P_1$ and  $G/ P_2$ defines a correspondence between  $G/P_1$ and  $G/ P_2$. More details on the Penrose correspondence and its applications can be found in the book \cite{EastP}.

The curved analogue of this construction for the parabolic Cartan geometries was intensively studied by Andreas \v Cap \cite{CapCor, CapTwo}. To be precise, we should say that a correspondence space construction deals only with one part of Penrose correspondence. Namely, let $G$ be a semisimple Lie group with two parabolic subgroups $P_2 \subset P_1 \subset G$. Consider a parabolic Cartan geometry  $(\G \rightarrow N, \ww)$ of type  $(G, P_1)$, where $\G$ is a principal $P_1$-bundle. 
\begin{defi}
A \emph{correspondence space}  of a parabolic geometry  $(\G \rightarrow N, \ww)$ is the orbit space $\CN=\G/P_2$.
\end{defi}  
The orbit space $\CN$ is a smooth manifold. A parabolic Cartan geometry $(\G \rightarrow \CN, \ww)$ of a type $(G, P_2)$ is naturally defined  on the correspondence space $\CN$. Moreover, curvature function $k^{\CN}$ of the correspondence space geometry is specified by the formula $k^{\CN}=j \circ k^N$, where $j$: $\Hom(\we ^2 \g/\gp_1, \g) \rightarrow \Hom(\we ^2 \g/\gp_2, \g)$ is a natural inclusion. 

\subsection{Extension Functor}

Extension functor is a type of construction which helps to describe various relations between Cartan geometries. General theory of extension functors can be found in \cite{CapSloBook}.

Let  $(M, P, \om)$ be an arbitrary Cartan geometry of type $(G,H)$. If $K$ is a closed subgroup of a group  $L$, we can construct a Cartan geometry of type $(L,K)$ over the same manifold $M$ with the help of the following data:
\begin{itemize}
\item
group homomorphism $i:H \rightarrow K$;

\item linear $H$-invariant  map $\al: \g \rightarrow \gll$, which is agreed  with $i$, i.e.  $\al \circ \Ad_H = \Ad_{i(H)} \circ \al$ and $\al|_\gh =  d i$
 \end{itemize}
In order to define new Cartan geometry over the same manifold we require an induced map $\al: \g/\gh \rightarrow \gll/\gk$ to be a linear isomorphism. We define a  principle $K$-bundle $\tilde{\PP} = \PP \times_i K $ . Let $j: \PP \rightarrow \tilde{\PP}$ be an inclusion defined by the formula $j(u) = (u,e)$. Then there  exists a unique Cartan connection $\om_\al : \tilde{\PP} \rightarrow  \gll $, for which we have $j^* \om_\al = \al \circ \om $. If $i$ and $\al$ are monomorphisms, then we can  think about  $\om_\al$ as a $K$-equivariant prolongation of the connection $\om$.

The curvature of extended geometry has simple relation with  initial geometry. Assume that  $\OO$ and  $\OO_\al$ are curvature tensors of Cartan connections $\ww$ and $\ww_\al$ respectively. Then we have:
\begin{multline*} 
j^* \OO_\al - \al \circ \OO = \dd(j^* \circ \om_\al - \al \circ \om) +  [j^*\om_\al,j^*\om_\al] - \al \circ [\om,\om] \\
= [\al(\om),\al(\om)] - \al \circ [\om , \om] = R_\al(\om) 
\end{multline*} 
Tensor $R_\al$ shows how linear the map $\al$ is far from Lie algebra homomorphism. 

With the help of correspondence space and extension functor construction we describe the geometry of conformal geodesics equations in the following section.

\section{Geometry of conformal geodesics equations}

Consider a flat conformal geometry of the dimension $m+1$. As usual, we define a quadratic form $q_{L}(x)$ on Lorenzian space $L=R^{m+3}$ by the formula
\[
q_L(x)= -2x_0 x_{m+2}-\sum_{i=1}^{m+1}x_i^2.
\]
The vector $V$ is called {\it light-like} if $q_L(V)=0$. A space $N$ of light-like points in $PL$ is called \textit{Mobius space}. This space is a homogeneous manifold:
\[
N=SO_{m+2,1}/P,
\]
where  $P$ is a stabilizer  of a light-like vector in $P L$. Lie algebras  $\so_{m+2,1}$ and $\gp$ of groups $SO_{m+2,1}$ and $P$ respectively have the following form:
\begin{equation}\label{l1}
 \so_{m+2,1} = \left(
\begin{array}{rrrrr}
h & y & q & 0 \\
x & 0 & -s^t & y \\
p & s & r & q^t \\
0 & x & p^t & -h
\end{array}
\right) 
; \qquad
 \gp = \left(
\begin{array}{rrrrr}
h & y & q & 0 \\
0 & 0 & -s^t & y \\
0 & s & r & q^t \\
0 & 0 & 0 & -h
\end{array}
\right)
\end{equation}
The only non-scalar blocks in \eqref{l1} here are $p$, $r$ and $q$. They have dimensions $m \times 1$, $m \times m$ and $1 \times m$ respectively.

\begin{prop}
The second jet space of unparameterized curves of the Mobius space $N$ is a homogeneous space $SO_{m+2,1}/P_2$, where the Lie group $P_2$ is determined by the following Lie algebra:
\[
 \gp_2 = \left(
\begin{array}{rrrrr}
h & y & 0 & 0 \\
0 & 0 & 0 & y \\
0 & 0 & r & 0 \\
0 & 0 & 0 & -h
\end{array}
\right).
\]
\end{prop}
\begin{proof}
The tangent space at a point $p$ can be identified with $\so_{m+2,1}/\gp$. The action of the group $P$ on  $TN_p$ can be realized as an adjoint action on $\so_{m+2,1}/\gp$. We denote a group, which preserves an chosen point in $TN$, as $P_1$. The group $P_1$ has the following Lie algebra:
\[
 \gp_1 = \left(
\begin{array}{rrrrr}
h & y & q & 0 \\
0 & 0 & 0 & y \\
0 & 0 & r & q^t \\
0 & 0 & 0 & -h
\end{array}
\right).
\]
The Lie group $SO_{m+2,1}$ acts transitively on the first jets of Mobius space $N$, therefore the first jets of curves form a homogeneous space $SO_{m+2,1}/P_1$. If we repeat the same construction for the $SO_{m+2,1}/P_1$ we get that the a stabilizer $P_2$ of a point in the second jet space is exactly $P_2$. The fact that the group $P_2$ acts transitively on the 2nd jet space ends the prove.
\end{proof}

Denote as  $\pi \colon \PP \to M$  a principle $P$-bundle with the Cartan form $\om\colon T\PP\to\so_{m+2,1}$. 
We define a new bundle $\overline{\pi}:\PP\to\wi M$ with the same total
space $\PP$ and a new base $\wi M = \PP/P_2$, i.e. the points of $\wi M$ are orbits of the right action of $P_2$ on $\PP$.

With the help of the above construction we introduce a curved analogue of the second jet space for conformal geometry. Next step is to extend this geometry to the Cartan geometry of the system of equations on conformal geodesics of conformal geometry $\pi$.

As above, let $ (G,H)$ be a Cartan geometry type of the system of $m$ ODEs of the third order. Denote a Lie algebra of $G$ as $\g$ and a Lie algebra of $H$ as $\gh$.
 Define an injective map  $\al : \so_{m+2,1} \to \g$  by the formula:
\[
\alpha \left( \left(
\begin{array}{rrrrr}
h & y & q & 0 \\
x & 0 & -s^t & y \\
p & s & r & q^t \\
0 & x & p^t & -h
\end{array}
\right)\right)=\left( \begin{array}{rr} -\frac{1}{2} h &  x \\  \frac{1}{2}  y & \frac{1}{2} h\end{array} \right) +r + \left(v_0 \otimes p - v_1 \otimes s  +v_2 \otimes q \right), \]
where the element $r\in\so_m$ is included naturally into $\glm$ and $ v_0 \otimes p$, $v_1 \otimes s$, $v_2 \otimes q$  is elements of $V_2 \ot W$.

The map $\al$ is not a Lie algebra homomorphism. The map 
\[ R_\al(x,y)=[\al(x),\al(y)]-\al([x,y]), \]
indicates how much the map $\al$ differs from a homomorphism of Lie algebras. Direct  computation shows that:
\begin{prop}\label{14}
The skew-symmetric tensor $R_\al$ takes nonzero values only on the following elements:
\begin{align*}
 &R(s_1,s_2)=s_1s_2^t-s_2s_1^t \in \glm ,\\
 &R(p_1,q_2)=-p_1q_2+q_2^tp_1^t+\left( \begin{array}{cc} -\frac{1}{2}q_2p_1 & 0 \\  0 & \frac{1}{2} q_2p_1\end{array} \right) \in \glm+\sl_2, \\
 &R(p_1,s_2)=\left( \begin{array}{cc} 0 & p_1^ts_2 \\  0 & 0\end{array} \right)\in\sl_2,\\
  &R(q_1,s_2)=\left( \begin{array}{cc} 0 & 0 \\  \frac{1}{2} q_1s_2 & 0\end{array} \right)\in\sl_2.
\end{align*}
\end{prop}

The restriction $\al|_{\gp_2}$ is a Lie algebra monomorphism from $\gp_2$ to $\gh$. This monomorphism defines an inclusion of Lie groups $i:P_2\to H$. We have all required information to define an extension geometry of type $(G,H)$ over the manifold $\wi M$.

Let us define an $H$-principal bundle $\wi \pi:\wi\PP\to \wi M$ by the formula $\wi\PP=\PP\times_i H$. The canonical inclusion $j\colon\PP\to \wi \PP$ defines a unique Cartan connection $\wi w=\ww_\al$ on $\wi \pi$ such that $j^*\wi \ww=\al\circ\ww$. In fact, $\wi w$ is
 an $H$-equivariant prolongation of the connection $\al\circ\ww\colon T\PP\to\g$.
The following commutative diagram describes a geometric picture that we have:
\[
\begin{CD}
\PP     @=   \PP     @>>>  \wi \PP   \\
@VVPV      @VVP_2V     @VVHV \\
M     @<<< \wi M   @= \wi M
\end{CD}
\]

Now we need to prove that the construction defined  above  describes an inclusion of 
the conformal geometry $\pi$ to the geometry of ODE system of the third order defined by conformal geodesics of $\pi$. To do that, first of all, we should check if this construction defines a Cartan connection adapted to some equation and then prove that conformal geodesics of our conformal structure are in one to one correspondence with the solutions of this equation.

A general form of a Cartan connection on a principal $ H$-bundle $\wi\pi$ is
\[ \wi\ww=\wi\ww_{-3}^i v_0\ot e_i+ \wi\ww_{-2}^i v_1\ot e_i+ \wi\ww_{-1}^i v_2\ot e_i+  \wi\ww_x x+ \wi\ww_h h+ \wi\ww_j^i e_i^j+\wi\ww_y y . \]
 We define on the manifold $\wi M$ a distribution $C$ with a splitting $ V\oplus E$ by conditions:
\begin{equation}\label{5}
\begin{aligned}
 &\langle s^*\wi\ww_{-3}^i,s^*\wi\ww_{-2}^i,s^*\wi\ww_{-1}^i\rangle^\bot = E , \\
&\langle s^*\wi\ww_{-3}^i,s^*\wi\ww_{-2}^i, s^*\wi\ww_{x} \rangle^\bot =  V,
\end{aligned}
\end{equation}
where $s^*$ is an arbitrary section of the bundle $\wi\pi$. This definition does not depend on the choice of the section $s$. Indeed, let $s'=s\cdot k$ be another section, where $k:\wi M\to H$. By the definition of a Cartan connection we have:
\[ s'^* \wi\ww = \operatorname{Ad}k^{-1}\left(s^* \wi\ww\right) + k^*\Theta_H. \]
Since the spaces 
\begin{align*}
&\left(v_0\ot W+v_1\ot W+v_2\ot W\right)^*,\\ 
&\left(v_0\ot W+v_1\ot W+\R x\right)^*
\end{align*} 
 are $H$-invariant, we have that: 
\[ \begin{aligned}
 &\langle s'^*\wi\ww_{-3}^i,s'^*\wi\ww_{-2}^i,s'^*\wi\ww_{-1}^i\rangle^\bot = E , \\
&\langle s'^*\wi\ww_{-3}^i,s'^*\wi\ww_{-2}^i, s'^*\wi\ww_{x} \rangle^\bot =  V.
\end{aligned} \]
 
 Every differential equation of a finite type is defined by a pair of distributions $(E,V)$.  This pair leads to a nilpotent differential geometry of type $(\gm, G_0)$, where $\gm_{-1}$ is a direct sum  $E\oplus V$ and $G_0$ 
is a subgroup of a group $ G_0(\gm)$ of all grading preserving automorphisms of $\gm$, which preserve splitting of $\gm_{-1}=E \oplus V$. Let $\g=\g(\gm, G_0)$  be a universal prolongation of a pair $(\gm, G_0)$.  Tanaka was the first, who  proved that for a semisimple $\g$ there exists a functor from Category of $G_0$-structures, associated with nilpotent geometries, to the category of a Cartan connection of type $(G, G_{\geq 0})$~\cite{tan70}. Later Morimoto generalized this statement  for a broader class of Lie algebras~\cite{mor93}, which the symbol of a finite type differential equation belongs to~\cite{dkm}. However, obtaining a nilpotent geometry from a Cartan connection of a type $(G,G_{\geq 0} )$  is much simpler task.

\begin{lem}
Let $\om$ be a Cartan connection of an infinitesimal type  $(\g,\g_{\geq 0} )$ over a principal bundle $\PP \rightarrow M$, where  $\g=\g(m, G_0)$ is a universal prolongation of a pair $(\gm,G_0)$. Assume, that a Cartan connection $\om$ has a curvature function of positive degree. If  $s: M \rightarrow \PP$ is an arbitrary section, then a pullback $s^*\om_-$ defines a nilpotent geometry of the type  $(\gm,G_0)$ and  this geometry does not depend on the choice of a section. 
\end{lem}
\begin{proof}
Consider a basis  $e^j_{-i}$, $i=1 \ldots n$, $j=1 \ldots n_i$ of a Lie algebra $\gm$, where the degree of $e^j_{-i}$ is $-i$. Then $s^*\om_{-}=\sum \om^j_{-i}e^j_{-i}$ and $\om^j_{-i}$ is a coframe on $M$. Let $X^j_{-i}$ be a dual frame.  Then we can obtain commutation relations between $X^j_{-i}$ from the  Cartan formula:

\begin{multline}\label{l3} \om^{j_1}_{-i_1} ([X^{j_2}_{-i_2},X^{j_3}_{-i_3}]) = -d\om^{j_1}_{-i_1} ([X^{j_2}_{-i_2},X^{j_3}_{-i_3}])+\\  X^{j_2}_{-i_2}(\om^{j_1}_{-i_1}(X^{j_3}_{-i_3})) +  X^{j_3}_{-i_3} (\om^{j_1}_{-i_1}(X^{j_2}_{-i_2})) = -d\om^{j_1}_{i_1} (X^{j_2}_{-i_2},X^{j_3}_{-i_3}) = \\(\frac{1}{2} [s^*\om, s^* \om] - s^* \Omega)^{j_1}_{-i_1}(X^{j_2}_{-i_2},X^{j_3}_{-i_3})
\end{multline}

We should check that \eqref{l3} is equal to structure constants of  $\gm$ whenever $i_1\geq i_2+ i_3$.
Note, that  $s^* \Omega$ does not affect the above property. This follows from the positivity of the curvature function.

Let's define a grading  on $X_{-i}^j$ by $\deg X^j_{-i}=-i$. $s^*\om$ is a non-negative function regarding this grading and the grading of $\g$.  Only $s^*\om_{-}$ part has a zero degree, while $s^*\om_{\geq 0}$ is always positive. It follows that \eqref{l3} is equal to structure constants of $\gm$ for $i_1 \geq i_2 + i_3$. This means that  $\gr X = \bigoplus_i X^*_{-i}/ X_{-i+1}$ is isomorphic to Lie algebra $\gm$.
\end{proof}

Now the existence of  the nilpotent geometry to which the Cartan connection is adapted and therefore the existence of an equation will follow from the positivity of the structure function. We are going to prove more strong fact: a normal conformal connection maps to a characteristic connection,  adapted to the third order ODEs system. Recall from \cite{medv2} the definition of characteristic Cartan connection.

The curvature of the Cartan connection $\wi\omega$ has the following form:
\[ \wi{\Omega}=\wi\Omega_{-3}^i v_0\otimes e_i+ \wi\Omega_{-2}^i v_1\otimes e_i+
\wi\Omega_{-1}^i
v_2\otimes e_i+  \wi\Omega_x x+ \wi\Omega_h h+ \wi\Omega_j^i e_i^j+\wi\Omega_y y .\]

Let $\Omega$ be one of the 2-forms $\wi\Omega_{-3}^i , \wi\Omega_{-2}^i ,
\wi\Omega_{-1}^i,
 \wi\Omega_x , \wi\Omega_h ,\wi\Omega_j^i,\wi\Omega_y$.
We can write it explicitly as: 
\[\Omega=\sum_{p,q=1}^3\Omega[\wi\omega_{-q}^j,\wi\omega_{-p}^k]\wi\omega_{-q}^j \wedge \wi{\omega}_{-p}^k+ \sum_{p=1}^3 \Omega[\wi\omega_x,\wi\omega_{-p}^k]\wi\omega_x \wedge \wi{\omega}_{-p}^k.\]
Then $\Omega[\wi\omega_{-q}^j,\wi\omega_{-p}^k]$ and $ \Omega[\wi\omega_x,\wi\omega_{-p}^k]$ are the coefficients of the structure function of the Cartan connection $\wi\omega.$
\begin{defi}
Cartan connection adapted to the equation 
\eqref{1} is called characteristic if the following conditions on a curvature is satisfied:
\begin{itemize}
 \item all coefficients  of degree $\le 1$ is equal to $0$;
 \item in degree $2$:  $\wi\Omega_h[\wi\omega_x\wedge \wi\omega_{-1}^i]=0$,  $\wi\Omega^i_j[\wi\omega_x\wedge
\wi\omega_{-1}^k]=0$,  $\wi\Omega_x[\wi\omega_x\wedge \wi\omega_{-2}^i]=0$, $\wi\Omega_{-1}^i [\wi\omega_x\wedge
\wi\omega_{-2}^j]=0$ and $\tr\left(\wi\Omega_{-2}^i [\wi\omega_x\wedge
\wi\omega_{-3}^j]\right)=0$;
 \item in degree $3$: $\wi\Omega_y[\wi\omega_x\wedge \wi\omega_{-1}^i]=0$, $\wi\Omega_h[\wi\omega_x\wedge
\wi\omega_{-2}^i]=0$, $\wi\Omega^i_j[\wi\omega_x\wedge\wi\omega_{-2}^k]=0$;
\item in degree $4$: $\wi\Omega_y[\wi\omega_x\wedge \overline\omega_{-2}^i]=0$.
\end{itemize}
\end{defi}

 It is quite rare when maps between Cartan geometries send a normal geometry to a normal one. Our construction also doesn't preserve normality. But we still have the following statement.

\begin{thm}\label{t1}
Assume that a conformal Cartan connection $\ww$ is normal. Then the extended Cartan connection $\wi\ww$ is characteristic.
\end{thm}
\begin{proof}
To prove the theorem we  analyze the curvature function of the Cartan connection $\wi \ww$. Let $k$ be a curvature function of the conformal connection $\ww$. Recall that a conformal connection is normal if $k_p=0$, $k_x=0$, $k_z=0$ and $k_{s,r}$ component lies in the kernel of the Ricci homomorphism. 

A curvature function $\wi k$ is an $H$-equivariant prolongation of
\[j^*\wi k = \al\circ k + R_\al.\]
The curvature function $\wi k$ is characteristic iff $\wi k$ take values in an $H$-invariant characteristic submodule $U$. Therefore, if 
$\al\circ k + R_\al$
take values in $U$ then $\wi k$ is characteristic.

 Note, that the tensor $R_\al$ is of the degree 4 and takes values in $I_4$ part of the curvature function. Moreover, the summand $\al\circ k$ doesn't contribute into $I_4$ part of the curvature function. That's why $I_4$ is equal to an $H$-equivariant prolongation of $R_\al$. From \cite{medv1} we know that the invariant $I_4$ takes values in $S^2(W)$. It can be checked directly the value of $R_\al$ on $\PP$ is equal to the bilinear form $E$, which has an identity matrix in the basis $(e_1,\dots,e_n)$.

Since on $s$ and $q$ parts of the Lie algebra $\som$ structure function of the connection $\ww$ are zero, all coefficients of the curvature function $\wi k $ except $\wi\Om^i[\wi\ww_x \we\wi\ww_{-3}^j]$ and $\wi\Om^i[\wi\ww_{-3}^k \we\wi\ww_{-3}^j]$ are zero. The form $\wi\Om^i$ here is one of the 2-forms $\wi\Omega_{-3}^i ,$ $\wi\Omega_{-2}^i ,$ $
\wi\Omega_{-1}^i,$ $
 \wi\Omega_x ,$ $ \wi\Omega_h ,$ $\wi\Omega_j^i.$ The lowest possible degree of $\wi\Om^i[\wi\ww_x \we\wi\ww_{-3}^j]$ and $\wi\Om^i[\wi\ww_{-3}^k \we\wi\ww_{-3}^j]$ is 2. The only such components of the degree 2 or 3 are $\wi\Omega_{-2}^i[\wi\ww_x \we\wi\ww_{-3}^j]$ and $\wi\Omega_{-1}^i[\wi\ww_x \we\wi\ww_{-3}^j]$. A curvature function with such components in degree 2 and 3 belongs to the characteristic module $U$ if the trace of $\wi\Omega_{-2}^i[\wi\ww_x \we\wi\ww_{-3}^j]$ is 0. In degree 4 characteristic connection has only one condition $\wi\Om_y [\wi\ww_x \we\wi\ww_{-2}^j]=0$, which is obviously satisfied for extended Cartan connection. 
 
The last thing that we should check is that the trace of component of the curvature which corresponds to Wilczynski invariant of degree 2 is $0$. We should use the fact that  $k_{s,r}$ component of the curvature of $\ww$ lies in the kernel of the Ricci homomorphism. The basic fact about Ricci homomorphism homomorphism is that elements
\[ b_{ij}=\sum_{k=1}^n \left( e^*_i \we e^*_k \ot e_k^j +e^*_j \we e^*_k \ot e_k^i\right)\]
are linearly independent and Ricci homomorphism induces an isomorphism from $B=\operatorname{span}(b_{ij})$ to $S^2(W_{n+1}^*)$. Under our extension construction an element $\frac{1}{2}b_{00}$ goes directly to the the trace of second Wilczynski invariant which will be zero for Ricci-flat connections.
\end{proof}
\begin{cor}
Invariant $I_2$ is equal to zero for conformal geodesics equations. Invariant $I_4$ is  a non-degenerate everywhere bilinear form for them.
\end{cor} 
\begin{proof}
The invariant $I_2$ appears in the coefficient $\wi\Om^i_{-2}[\wi\ww^k_{-1}\we\wi\ww^j_{-3}]$, which is equal to zero for equations on conformal geodesics. The second statement follows from the proof of the previous theorem.
\end{proof}

\begin{thm}
Let us consider a conformal geometry with a normal Cartan connection $\ww:\PP\to\so_{m+1,1}$. Then the extended connection $\wi\ww:\PP\times_i H\to\g$ is adapted to the ODEs system on conformal geodesics of  $\ww$. 
\end{thm}
\begin{proof}

A symmetry algebra $\gs$ of a circle in $m+1$ dimensional euclidean space is $\sl_2\times\so_{m-1}$. One can realize this algebra as subalgebra of $\so_{m+1,1}$: 
\[ {\mathfrak s} = \left(
\begin{array}{rrrrr}
\wi{h} & \wi{y} & 0 & 0 \\
\wi{x} & 0 & 0 & \wi{y} \\
0 & 0 & r & 0 \\
0 & \wi{x} & 0 & -\wi{h}
\end{array}
\right) . \]
Curve $\wi\gamma$ is a circle if and only if there exists such a section $s\colon SO_{m+1,1}/P \to SO_{m+1,1}$ that $s^*\om_G|_\gamma$ takes values in $\gs$, where $\om_G$ is Maurer-Cartan form.
By the definition of the development the curve $\gamma$ on the manifold $M$ is conformal geodesic iff there exists at every point a local section $s$ such that
$s^*\om|_\gamma$ takes values in $\gs$, where $\om$ is a normal conformal Cartan connection.

 The property  $s^*\om|_\gamma\in\gs$ is equal to the property 
\[
s^*\om_{p}|_\gamma=s^*\om_{q}|_\gamma=s^*\om_{s}|_\gamma=0.
\]
 
 On the other hand the curve $\gamma$ on the manifold $\wi M$ is a solution of the equation \eqref{1} iff
$\wi\om_{-1}$, $ \wi \om_{-2}$, $\wi\om_{-3}$ are equal to zero on the curve.

Indeed, let $\oo$ be the following co-frame on $\E$
\begin{align*}
\oo_x&= d x \\
\oo_{-3}^i&=d y^i-p^i d x, \\
\oo_{-2}^i&=d p^i-q^i d x ,\\
\oo_{-1}^i&=d q^i-f^i d x.
\end{align*}
The coframe $\wi\om_{-1}$, $ \wi \om_{-2}$, $\wi\om_{-3}, \wi\om_x$ is adapted to the equation \eqref{1} if and only if
\begin{align*}
 &\langle \wi\ww^i_{-3} \rangle = \langle \oo^i_{-3} \rangle \\
 &\langle \wi\ww^i_{-3},\wi\ww^i_{-2} \rangle = \langle \oo^i_{-3},\oo^i_{-2} \rangle \\
&\langle \wi\ww^i_{-3},\wi\ww^i_{-2},\wi\ww^i_{-1} \rangle = \langle \oo^i_{-3},\oo^i_{-2},\oo^i_{-1} \rangle \\
&\langle\wi \ww^i_{-3},\wi\ww^i_{-2},\wi\ww_{x} \rangle =\langle \oo^i_{-3},\oo^i_{-2},\oo_{x} \rangle
\end{align*}
We see that a curve $\gamma$ on the manifold $\wi M$ is a solution of the equation \eqref{1} iff
$\oo_{-1}$, $ \oo_{-2}$, $\oo_{-3}$ are equal to zero on the curve $\gamma$. This is equivalent to the fact that $\wi\om_{-1}$, $ \wi \om_{-2}$, $\wi\om_{-3}$ are equal to zero on the curve $\gamma$.
\end{proof}

\section {Necessary and sufficient conditions on conformal geodesics equations}

In this section we obtain the necessary and sufficient conditions for the 3rd order ODEs systems, which define conformal geodesics. Since we know how to construct a Cartan connection of conformal geodesics ODEs system from conformal connection, we revert this process and obtain conditions on a curvature, which can be formulated in terms of fundamental invariants.

From the proof of the Theorem \ref{t1} we know that  for conformal equations the value of the invariant $I_4$ over the bundle $\PP$ is an identity bilinear form $E\in S^2(W^*)$. The group $i(P_2)$ acts trivially on $I_4=E$. The bundle $\PP$ is a maximal subbundle of $\wi\PP$ on which the invariant $I_4$ is identity.

 We define the reduction of the bundle $\tilde{\PP}$ on which $I_4$ is equal to $E$ as $\tilde{\PP}_E$. This is a reduction to the group $P_2.$ Let $\ww_E$ be a restriction of the form $\tilde{\ww}$ to the subbundle $\PP_E$. In  order to make the form $\ww_E$ to define a geometry of the type  $(SO_{m+2,1},P_2)$  we need form $\ww_E$ to take values in $\im_\al(\so_{m+2,1}).$ This condition can be formulated in terms of differential relations between fundamental invariants.

Recall, that a universal covariant derivative of a function $f\colon \PP \to \overline{V} $  is a function $Df\colon \PP \to  \overline{V} \otimes \g^*$. Since $H$-module $\g^*$ decomposes as $\glm^* \oplus \sl^*_2 \oplus V^*$ we have a decomposition of $D$ to $D_{\glm} + D_{\sl_2} + D_V$.

\begin{prop}
The connection $\tilde{w}_E$ takes values in the $\im_\al(\so_{m+1,1})$ iff the covariant derivatives  $D_{\sl_2} I_4$ and $D_V I_4$ are equal to $0$.
\end{prop}
\begin{proof}
Let $X_i$ be a basis of $\g$. Then the Cartan connection  $\tilde{w}$ has a form $\sum X_i \tilde{\ww}^i$ in this basis. According to the definition of universal covariant derivative
\begin{equation}\label{l11}
d I_4 = \sum D_{X_i} I_4 \tilde{\ww}^i.
\end{equation}

In our case $I_4$ equals to $E$ on $\PP$. Let $j$ be an inclusion $j: \tilde{\PP}_E \hookrightarrow \tilde{\PP}.$ Assume now that $\tilde{w}_E=j^*\tilde\ww \in \Hom (T\PP, \so_{m+2,1})$. Then from \eqref{l11} we get
\begin{equation} \label{l12}
0=j^*dI_4 = \sum_{X_i \in \g} j^*(D_{X_i} I_4) j^*\tilde{\ww}^i=\sum_{X_i \in \so_{m+2,1}} j^*(D_{X_i} I_4) j^* \tilde{\ww}^i
\end{equation}

We claim, that $j^*\tilde{\ww}^i$ form a coframe on $\PP_E$ for $i$ such that $X_i\in\so_{m+2,1}$. To show this, assume that $j^*\tilde{\ww}^i$ do not form a coframe. Since $\dim T\PP_E=\dim \so_{m+2,1}$ there exists a tangent vector $v$ at some point $p$ of ${\PP_E}$ such that $j^*\tilde{\ww}(v)=0.$ But  $\tilde{\ww}$ is an $H$-equivariant prolongation of $j^*\tilde{\ww}$ therefore $\tilde{\ww}(v)=\Ad_H\circ j^*\tilde{\ww}(v)=0$. This contradicts with the fact that $\tilde{\ww}$ is a Cartan connection.

Now from the equation \eqref{l12} it follows that $j^*(D_{X_i} I_4)=0$ for $X_i\in\im_\al(\so_{m+1,1}).$ Since $D_{\sl_2} I_4$ and $D_{V} I_4$ are $H$-equivariant prolongations of   $j^*(D_{\sl_2} I_4)$  and $j^*(D_{V} I_4)$ we get that  $D_{\sl_2} I_4$ and  $D_{V} I_4$ are equal to $0$.

And vice versa, if covariant derivatives  $D_{\sl_2} I_4$ and $D_V I_4$ are equal to zero then the connection $\tilde{w}_E$ takes values in the $\im_\al(\so_{m+1,1})$. The covariant derivative $D_{X_i}(E)$ is equal to $-2X_i$ for the operators with symmetric matrix $X_i \in \glm$. Then using \eqref{l11} we get:
 \begin{multline} \label{l13}
0=j^*dI_4 = \sum_{X_i \in \g} D_{X_i} j^*I_4 j^*\tilde{\ww}^i=\\\sum_{X_i \in \Sym\glm} (D_{X_i} I_4)j^* \tilde{\ww}^i=
\sum_{X_i \in \Sym\glm} (-2{X_i} I_4)j^* \tilde{\ww}^i.
\end{multline}
Since basis operators $X_i\in\Sym\glm$ are linearly independent, all $j^* \tilde{\ww}^i$ should  be equal to zero for such $i$. This means that $\ww_E=j^*\tilde\ww \in \Hom (T\PP, \so_{m+2,1})$.
\end{proof}

\begin{thm}\label{t3}
The 3rd order ODEs system determines locally conformal geodesics of some conformal geometry iff the following conditions are satisfied:
\begin{enumerate}
\item\label{t3l1} Invariant $I_2$ equal to zero;
\item\label{t3l2} Invariant $I_4$ has the maximal rank and $D_{\sl_2}(I_4)=0$, $D_{V_2\ot W}(I_4)=0$;
\item\label{t3l3} $i_{\zeta}(k - I_4)=0$ for $\zeta \in v_1\ot W+v_2\ot W$,
\end{enumerate}
where $k$ is a structure function of the adapted to the equation Cartan connection.
\end{thm}
\begin{proof} 
At the current point we obtain the connection $\ww_E$ and the principal $P_2$-bundle which could be a correspondence space for a conformal structure. The theory of correspondence spaces gives the answer when the bundle $\PP_E$ comes from conformal structure \cite{CapCor}.  A Cartan connection $\ww_E$ on $\PP \to \wi{M}$ defines some geometry of type $(G,P)$ locally iff its curvature function $k_E$ vanishes on $\gp$.

We conclude, that $\ww_E$ induces locally some conformal structure iff $i_\zeta k_E=0$ for $\zeta $ from $\R s\oplus \R q\subset\so_{m+1,1}$. This means that the curvature function $k_E$ takes values in $v_0^*\ot W\we v_0\ot W^*\ot\gh$. But this space is an $P$-submodule and $\al$ in an $P$-homomorphism of modules. Therefore a function $\left(k-R_\al\right)$, which is $H$-equivariant prolongation of $\al(k_E)$, also takes values in $v_0^*\ot W\we v_0\ot W^*\ot\gh$. Here we use the fact that an $H$-module is also a $P$-module The reformulation of this condition is $\ww_E$ induces locally some conformal structure iff $i_\zeta\left( k-I_4\right)=0$ for $\zeta \in v_1\ot W+v_2\ot W$.
\end{proof}.

\section{Invariants of conformal geodesics equations}
In this section we provide explicit formulas which completely define a family of equations on conformal geodesics. The computations here are based on formulas for the characteristic connection of the 3rd order ODEs systems \cite{medv2}. We start with study of the family of equation, which is characterized by condition $I_2=0$. Then we compute relations on invariants of conformal geodesics equations.
 \subsection{Invariant $I_2$}
The equality to zero of relative invariants always determines a stable under point transformations family of equations. One of the earliest examples is due to Cartan \cite{CartanProj}. He prove that $f_{y'y'y'y'}$ is relative invariant of the one 2nd order equation of the form $y''=f(x,y,y')$. Therefor, a family of equations
\[ y''=A(x,y)\left(y'\right)^3+B(x,y)\left(y'\right)^3+
C(x,y)\left(y'\right)^3+D(x,y)\]
is invariant under the action of the diffeomorphism group of the plane. Moreover, Cartan shows that every such equation describes geodesics of normal projective structure.

The equations for which the condition $I_2=0$ is satisfied also define nice family of equations. 

\begin{thm}\label{cl1}
Invariant $ I_2 $ is equal to $0$ iff the 3rd order ODEs system has the form:
\[ f^i(q_j, p_k, y_l,x)=3q_i\sum_{j=1}^m \A_j(p_k, y_l,x)q_j + \sum_{j=1}^m \B_j (p_k, y_l,x)q_j +\C(p_k,y_l,x) .\]
\end{thm}
\begin{proof}
We don't use Einstein summation notion in this proof. Let  us denote $\frac{\d^2 f^i}{\d q^j \d
q^k}$ as ${I}^i_{jk}$. Since the tensor ${I}^i_{jk}$ is symmetric in indexes $j$ and $k$ it has 2 equal traces, which we denote as $T_j.$ 
The trace-free part of $I_2$ is the following:
\[{I}^i_{jk}-\frac{1}{n+1} (T_j\delta^i_k+T_k\delta^i_j)\]
  We prove now that functions $f^i$ are quadratic polynomials with respect to $q^i$. Indeed, from condition $\tr_0 I=0$ follows that ${I}^i_{jk}=0$ if $i\neq j$ and $i\neq k$.
 Therefore, all partial derivatives $\frac{\d^3 f^i}{\d q_j \d q_k \d q_l}$ are equal to $0$, if two indexes in the set $\{j,k,l\}$ are different from $i$. We differentiate equality 
\[\frac{\d^2 f^i}{\d q_i \d q_j}-\frac{1}{n+1}\sum_{k=1}^m \frac{\d^2 f^k}{\d q_k \d q_j}=0,\]
where $i\neq j$, by $q_i$ and get that $\frac{n}{n+1} \frac{\d^3 f^i}{\d^2 q_i \d q_j}=0$
  Similarly, from equation
\[ \frac{\d^2 f^i}{\d^2 q_i }-\frac{2}{n+1}\sum_{k=1}^m \frac{\d^2 f^k}{\d q_k \d q_i}=0\]
 follows that $\frac{\d^3 f^i}{\d^3 q_i}=0.$
  
  The functions $f^i$ could be expressed in the following form:
  \[f^i=q_i\left( \A^i_j q_j\right) + \B^i_j q_j + \C_i ,\]
  where coefficients $\A^i_j,$ $\B^i_j$ and $\C_i$ don't depend on $q_k.$
The condition $I^i_{ij}=0$ gives us that
\[
\A^i_j-\frac{1}{n+1}\sum_{k=1}^m \A^k_j=0.\]
Therefore, all coefficients $\A^i_j$ are the same for different $i.$ Let $3\A_j=\A^i_j$. Then from condition $I_2=0$ follows that $f^i$ has the form:
\begin{equation}\label{l05} f^i=3q_i\left( \A_j q_j\right) + \B^i_j q_j + \C_i.
\end{equation}
The direct computation shows that for every system of the form \eqref{l05} invariant $I_2$ is equal to 0.
\end{proof}
\begin{cor}
The the family of equations of the form:
\[ f^i(q_j, p_k, y_l,x)=3q_i\sum_{j=1}^m \A_j(p_k, y_l,x)q_j + \sum_{j=1}^m \B_j (p_k, y_l,x)q_j +\C(p_k,y_l,x) \]
is invariant under the group of point transformations.
\end{cor}
\subsection{Conditions on conformal geodesics equations}
 The end of the article is devoted to proving of the following theorem, which gives explicit conditions on  coefficients of conformal geodesics equations. The formulas below are quite complicated. I have a realization of this formulas in Maple, which I could send on demand.
\begin{thm}
A system of the 3rd ODEs is a system on conformal geodesics iff the following equations are satisfied:
\begin{align*}
0=&\tr_0\left(\frac{\d^2 f^i}{\d q^j \d
q^k} \right),\\0=&
-\frac{1}{3}\frac{\d f^i}{\d q^k}(I_4)_{kj}-\frac{1}{3}\frac{\d f^k}{\d q^j}(I_4)_{ik}-\frac{d}{dx}I_4,\\0=&
\frac{\d}{\d q^k}I_4,\\0=&
\frac{\d W_2}{\d q^k},\\0=&
\tr_0\left(\frac{\d (W_2)^i_k}{\d p^j}-2\frac{\d (W_3)^i_k}{\d q^j}+\frac{\d (W_3)^i_j}{\d q^k} -8H^{-1}_j(W_2)^i_k+2H^{-1}_k(W_2)^i_j\right), \\0=&
\frac{\d}{\d q^l}\left(-2\frac{\d (W_3)^i_k}{\d q^i}+\frac{\d (W_3)^i_i}{\d q^k}\right), \\0=&
-H^{-1}_l\left(\frac{\d (W_2)^i_k}{\d p^i}-2\frac{\d (W_3)^i_k}{\d q^i}+\frac{\d (W_3)^i_i}{\d q^k} -8H^{-1}_i(W_2)^i_k\right)+ \\
&+\frac{\d}{\d q^l}\left(D_{e^i_{-2}}(W_3)^i_j -D_{e^j_{-2}}(W_3)^i_i\right)+\left(I_4\right)_{ji}\left(W_2\right)^i_l,
\end{align*}
where
\[ D_{e^k_{-2}}W_3=\frac{\d W_3}{\d p^k}-\frac{\d W_3}{\d q^l}B^l_k-6H^{-1}_kW_3+2(H^{-1}_k-H^{-2}_lB^l_k)W_2+G^{*,-2}_{*k}W_3-W_3G^{*,-2}_{*k} \]
and $\tr_0$ is a trace-free part of the tensor in the indexes $i,j$.
\end{thm}
To proceed to direct computation, we recall first formulas for the characteristic Cartan connection of the 3rd order ODEs system \cite{medv2}.
We are working using a coframe $\theta$, which is defined in the following way:
\begin{align*}
  \theta_x & =dx,& \theta^i_{-1} & =dq^i-f^i(x,y,p,q)\,dx,&\\
  \theta^i_{-2} & =dp^i-q^i\,dx ,&
  \theta^i_{-3} & =dy^i-p^i\,dx,.
\end{align*}
 Consider a Cartan geometry of the type $(G,H)$ over an $H$-principle bundle $\PP\to\E$ with a Cartan connection $\tilde{\ww}\colon T\PP\to \g$:
 \[
 \tilde{\omega }=\tilde{\omega }^i_{-3} v_0\otimes e_i +
 \tilde{\omega }^i_{-2} v_1\otimes e_i +
 \tilde{\omega }^i_{-1} v_2\otimes e_i + \tilde{\omega }_x x + \tilde{\omega }_h h +
\tilde{\omega }^i_j e^j_i + \tilde{\omega }_y y .
\]
  We  fix a gauge for the connection $\tilde{\omega } $ with the help of a section $s\colon\E\to\tilde\PP$, which is uniquely determined with the following formulas:
 \begin{align*}
 s^* \tilde{\omega}_{-3}^i &=\theta_{-3}^i ,\\
 s^* \tilde{\omega}_h &\equiv 0 \mod \langle \theta_{-3}^i,\theta_{-2}^i,\theta_{-1}^i \rangle ,\\
 s^* \tilde{\omega}_x &\equiv -\theta_x \mod \langle \theta_{-3}^i,\theta_{-2}^i,\theta_{-1}^i \rangle.
\end{align*}
A pullback  $\omega \colon T\E \to \g$ is defined by the formula $ \omega  = s^* \tilde{\omega }$.
According to \cite{medv2}, a characteristic Cartan connection adapted to the equation \eqref{1} has the form:
\begin{align*}
\omega_{-3}^i &=\theta_{-3}^i ,\\
\omega_{-2}^i &=\theta_{-2}^i + A_j^{i}\theta_{-3}^j,\\
\omega_{-1}^i &=\theta_{-1}^i + B_j^{i}\theta_{-2}^j+C_j^{i}\theta_{-3}^j,\\
\omega_x &=-\theta_x +E_j\theta_{-3}^j , \\
\omega_h &=F^{-2}_j\theta^j_{-2}+ F^{-3}_j\theta^j_{-3},\\
\omega^i_j &=G^{i,x}_j \theta_x +G^{i,-2}_{jk}\theta^k_{-2}+G^{i,-3}_{jk}\theta^k_{-3},\\
\omega_y &=H^x \theta_x + H^{-1}_j\theta^j_{-1}+H^{-2}_j\theta^j_{-2}+ H^{-3}_j\theta^j_{-3},
\end{align*}
with the following coefficients:

\begin{align*} 
A_j^i&=-\frac{1}{3}\frac{\d f^i}{\d q^j} ,\\
 B^i_j&=-\frac{2}{3}\frac{\d f^i}{\d q^j}, \\
  C^i_j&=-\frac{\d f^i}{\d p_j} +\frac{2}{3}\frac{d}{dx}\frac{\d f^i}{\d q^j}
-  \frac{2}{9} \frac{\d f^i}{\d q^k}\frac{\d f^k}{\d q^j}-2H^x\delta^i_j ,\\
E_j&= -\frac{1}{3(m+1)}\frac{\d^2 f^i}{\d q^j\d q^i},\\
F^{-2}_j&=\frac{1}{6(m+1)}\frac{\d^2 f^i}{\d q^j\d q^i},\\ 
 F^{-3}_j &=\frac{\d H^x}{\d q^j} -
H^{-1}_k  \frac{\d f^k}{\d q_j} -\frac{1}{3(m+1)}\frac{d}{d x}\frac{\d^2 f^i}{\d q^j\d q^i},\\ 
G^{i,x}_j&=-\frac{1}{3}\frac{\d f^i}{\d q^j}, \\
 G^{i,-2}_{jk} &=-\frac{1}{3}\frac{\d^2 f^i}{\d q^j\d q^k},\\
  G^{i,-3}_{jk} &= -\frac{1}{3}\frac{\d^2 f}{\d p^k\d q^j} - \frac{d}{dx} G^{i,-2}_{jk} - G^{i,x}_l G^{l,-2}_{jk} + G^{i,-2}_{lk}  G^{l,x}_j , \\
  H^x&=\frac{1}{4m}\left(-\frac{\d f^i}{\d p_i} +\frac{d}{dx}\frac{\d f^i}{\d q^i}
-  \frac{1}{3} \frac{\d f^i}{\d q^k}\frac{\d f^k}{\d q^i}  \right), \\
H^{-1}_j&= \frac{1}{6(m+1)}\frac{\d^2 f^i}{\d q^j\d q^i},\\ 
 H^{-2}_j& = \frac{\d H^x}{\d q^j} - \frac{d }{d x}H^{-1}_j -
H^{-1}_k  \frac{\d f^k}{\d q^k}, \\
 H^{-3}_j &= \frac{\d H^x}{\d p^j} - \frac{d H^{-2}_j}{d x} -
H^{-1}_k  \frac{\d f^k}{\d q^j} - 2H^xH^{-1}_j
\end{align*}

Recall that there exist 4 fundamental invariants for the system of the third order ODEs. The list of the invariants is the following:
\begin{align*} 
\left(W_2\right)^i_j =& \tr_0 \left(\frac{\d f^i}{\d p^j} -
\frac{d}{dx} \frac{\d f^i}{\d q^j} +
\frac{1}{3} \frac{\d f^i}{\d q^k}\frac{\d f^k}{\d q^j} \right), \\
\left(I_2\right)^i_{jk} =& \tr_0\left(\frac{\d^2 f^i}{\d q^j \d
q^k}
 \right), \\
\left(W_3\right)^i_j =&
\frac{\partial f^i}{\partial y^j} +\frac{1}{3} \frac{\partial f^i}{\partial q^k}\frac{\partial f^k}{\partial p^j} - \frac{d}{dx} \frac{\partial f^i}{\partial
p^j} + \frac{2}{3}\frac{d^2}{dx^2} \frac{\partial f^i}{\partial
q^j} + \frac{2}{27}(\frac{\partial f^i}{\partial q^j})^3 \\ 
&-\frac{4}{9}
\frac{\partial f^i}{\partial q^k}\frac{d}{dx}\frac{\partial f^k}{\partial q^j} 
 -
\frac{2}{9} \frac{d}{dx}\left(\frac{\partial f^i}{\partial q^k}\right)\frac{\partial f^k}{\partial q^j} - 2\delta^i_j\frac{d}{dx}H^x ,\\
 \left(I_4\right)_{jk} =& - \frac{\d H_k^{-1}}{\d p_j}+\frac{\d}{\d q^k}H^{-2}_j+2H_j^{-1} H_k^{-1},
\end{align*}  

\subsection{Covariant derivation}
Assume that $\wi f\colon\tilde\PP\to V$ is an $H$-equivariant function, i.e. $\wi f(g.p)=\Ad (g^{-1}).\wi f(p).$ If $\wi\ww$ is a Cartan connection on $\PP$ then a decomposition of $\wi\ww$ in some basis $X_i$ of $\g$:
\[ \wi\ww=\sum X_i \wi\ww^i \]
defines a coframe $\wi\ww^i$ on $\PP$. In this settings the covariant derivative along directions $X_i$ can be determined from the decomposition:
\[
  df=\sum \left(D_{X_i} \wi f\right) \wi\ww^i=\sum \wi f_i \wi\ww^i. 
\]

In the gauge, which is determined by the section $s$, we denote the pullback of the function $\wi f$ by $f$  and the pullback of the function $\wi f_i$ by $f_i$. If we assume that $\dim \g = n+m$ and $\dim \gh = n$, than 
\[
 df=-\sum_{i=1}^n\left( X_i.f \right)\ww^i+\sum_{i=n+1}^{n+m}f_i\ww^i, 
\]
 where $\ww^i=s^*\wi \ww^i$. Therefore, to compute a covariant derivative along the direction $X_i\in \g_-$ in the selected gauge we can use the formula:
\begin{equation}\label{l02} 
f_i= d f(X_i)+\sum_{j=1}^n\left( X_j.f \right)\ww^j(X_i).
 \end{equation}
 Now we apply the formulas above to the calculation of covariant of covariant derivatives of the invariant $I_4$. The full differential on the surface $\E$ has the form:
 \begin{equation}
  d f = \frac{d f}{dx}\oo_x+\frac{\d f}{\d y^i} \oo^i_{-3}+\frac{\d f}{\d p^i} \oo^i_{-2}+\frac{\d f}{\d q^i} \oo^i_{-1}.\label{l01}
\end{equation} 
Recall that $\langle e^i_{-3},e^i_{-2},e^i_{-1}\rangle = \langle v_0\ot e_i,v_1\ot e_i,v_2\ot e_i\rangle$ is the basis of $V_2\ot W$. Using formulas \eqref{l01}-\eqref{l02} and the \ref{t3l2} condition of the theorem \ref{t3} we obtain that the following conditions should be satisfied for the conformal geodesics equation:
  \begin{align}\label{l03}
  D_{x} I_4&=G^x.I_4+I_4.G^x-\frac{d}{dx}I_4=0,\\
  \label{l04}D_{e^i_{-1}} I_4&=\frac{\d}{\d q^i}I_4=0.  \end{align}

We prove that from the conditions \eqref{l03}-\eqref{l04} follows \[D_{\sl_2}=D_{V_2\ot W}(I_4)=0.\]

First, note that $D_{h}I_4=D_{y}I_4=0$ since the Lie algebra $\sl_2$ acts on $I_4\in S^2(W^*)$ trivially. Secondly, conditions $D_{e^i_{-2}} I_4=0$ and $D_{e^i_{-3}} I_4=0$ follow from this useful formula:
\begin{equation} \label{l07}
D_X D_Y - D_Y D_X = D_{[X,Y]}+D_{\OO(X,Y)}. 
\end{equation}

Indeed, $D_h$ and $D_y$ acts trivially on $I_4\in S^2(W^*).$  Coefficients $\OO^i_j(e^k_{-1},x)$ and $\OO^i_j(e^k_{-2},x)=0$ are equal to $0$ for characteristic connections, and $\OO^i_{-a}(e^j_{-1},x)=\OO^i_{-a}(e^j_{-2},x)=0,$ $a=1,2,3$ for characteristic connections with invariant $I_2=0$. Therefore $D_{\OO(e^j_{-1},x)}I_4=D_{\OO(e^j_{-2},x)}I_4=0,$ and
\[ D_{e^i_{-2}}I_4=D_{[x,e^i_{-1}]}I_4=D_x D_{e^i_{-1}}I_4 - D_{e^i_{-1}} D_x I_4=0,\] 
\[ D_{e^i_{-3}}I_4=D_{[x,e^i_{-2}]}I_4=D_x D_{e^i_{-2}}I_4 - D_{e^i_{-2}} D_x I_4=0.\]
\subsection{Correspondence conditions}  Now we proceed with the computations of formulas which express the \ref{t3l3}rd condition of the theorem \ref{t3} in terms of the fundamental invariants of the system.
\begin{prop}\label{p4}
The condition \ref{t3l3} of the theorem 3 is equivalent to the following conditions on Wilczynski invariants:
\begin{align}
&D_{e^k_{-1}}(W_2)^i_j=0,\label{p4l1}\\
&\tr_0\left( D_{e^k_{-1}}(W_3)^i_j - 2D_{e^j_{-1}}(W_3)^i_k+D_{e^j_{-2}}(W_2)^i_k\right)=0,\\
& D_{e^l_{-1}}\left( D_{e^k_{-1}}(W_3)^i_i - 2D_{e^i_{-1}}(W_3)^i_k+D_{e^i_{-2}}(W_2)^i_k\right)=0, \\
& D_{e^k_{-1}}\left( D_{e^i_{-2}}(W_3)^i_j -D_{e^j_{-2}}(W_3)^i_i\right)+I_4.W_2=0 \label{20},
\end{align}

 Here the operator $\tr_0$ is the trace-free part of a tensor with respect to the indexes $i,j.$
\end{prop}
\begin{proof}
  We use the Bianchi identity to compute explicit formulas for condition \ref{t3l3} in terms of fundamental invariants and their covariant derivatives. We denote the structure function of the Cartan connection $\ww$ as $C.$ The Bianchi identity can be written in the form
  \[ \d C=d C+C\circ C,\]
  where $\d$ is a Lie algebra differential, 
  \[d C=x^*\we D_xC+{e_{-1}^i}^*\we D_{e_{-1}^i}C+{e_{-2}^i}^*\we D_{e_{-2}^i}C+{e_{-3}^i}^*\we D_{e_{-3}^i}C \]
  is the part of universal covariant derivative of $C$ and
 \[ (C\circ C)(X_1,X_2,X_3)=C(C(X_{[1},X_2),X_{3]}).\]
 For torsion free geometries the term $C\circ C$ is zero. However, the connection $\ww$ has a torsion, that is why $C\circ C$ is not 0 in our situation. The first non-zero term of $C\circ C$ has the degree 6 in an assumption that invariant $I_2$ is equal to 0. Therefore, $C\circ C$ term affects only formula \eqref{20}. Present the structure function $C$ as a sum of the following form:
 \[
 C=C^i_{-3} e^i_{-3} + C^i_{-2} e^i_{-2} +
 C^i_{-1} e^i_{-1} + C_x x + C_h h +
C^i_j e^j_i + C_y y.
\]
We use the same notation for the Lie algebra differential of $C$:
 \[
 \d C=\d C^i_{-3} e^i_{-3} + \d C^i_{-2} e^i_{-2} +
 \d C^i_{-1} e^i_{-1} + \d C_x x + \d C_h h +
\d C^i_j e^j_i + \d C_y y.
\]

  In the 3rd degree the Bianchi identity gives us among the others relations the following one:
  \begin{align*}  \d C^i_{-3}(e^j_{-3},e^k_{-2},x)=&-C^i_j(e^k_{-2},x)+2C_h(e^k_{-2},x)\delta^i_j \\&+ C_x(e^j_{-3},x)\delta^i_k-C^i_{-3}(e^j_{-3},e^k_{-3})+C^i_{-2}(e^j_{-3},e^k_{-2}) \\&= C_x(e^j_{-3},x)\delta^i_k-C^i_{-3}(e^j_{-3},e^k_{-3})+C^i_{-2}(e^j_{-3},e^k_{-2}) =0.
  \end{align*} 
  Terms $C^i_j(e^k_{-2},x)$ and $ 2C_h(e^k_{-2},x)$ are equal to 0 due to normality conditions on the curvature. Assuming $i=j=k$, we get that: 
  \[C_x(e^i_{-3},x)=C^i_{-2}(e^i_{-3},e^i_{-2}).\]
 In the degree 3 we have only one non-zero covariant derivative:
 \begin{multline*} \d C^i_{-2}(e^k_{-1},e^j_{-3},x) =-2C_y(e^k_{-1},x)\delta^i_j+C^i_{-2}(e^j_{-3},e^k_{-2})-C^i_{-2}(e^j_{-3},e^k_{-2})-\\C_x(e^j_{-3},x)\delta^i_k 
 =C^i_{-2}(e^j_{-3},e^k_{-2})-C^i_{-2}(e^j_{-3},e^k_{-2}) -C^j_{-2}(e^j_{-3},e^j_{-2})\delta^i_k =D_{e^k_{-1}}(W_2)^i_j. 
\end{multline*} 
  From condition \ref{t3l3} of the theorem \ref{t3} follows that covariant derivative $D_{e^k_{-1}}W_2$ should be 0 for equations on conformal circles. The covariant derivative $D_{e^k_{-1}}W_2$ is the only fundamental invariant covariant derivative  of degree 3. Since $D_{e^k_{-1}}W_2=0$, all terms of the degree 3 in the curvature function apart from the Wilczynski invariant $W_3$ should be equal to 0 for conformal geodesics equations.
  
  In the degree 4 we are interested in the following parts of the Bianchi identity:
  \begin{align*}
  \d C^i_{-2}(e^k_{-2},e^j_{-3},x) =& C^i_{-1}(e^k_{-2},e^j_{-3})-C^i_{-2}(e^k_{-3},e^j_{-3})-C_k^i(e^j_{-3},x)=D_{e^k_{-2}}(W_2)^i_j,\\
  \d C^i_{-1}(e^k_{-1},e^j_{-3},x) =& 2\delta^i_k C_h(e^j_{-3},x)-C^i_{-1}(e^k_{-2},e^j_{-3})-C_k^i(e^j_{-3},x)=D_{e^k_{-1}}(W_3)^i_j, \\
  \d C^i_{-1}(e^k_{-2},e^j_{-2},x) =& -C^i_{-1}(e^k_{-2},e^j_{-3}) -C^i_{-1}(e^k_{-3},e^j_{-2})=0,\\
  \d C^i_{-1}(e^k_{-1},e^j_{-3},x) =& 2\delta^i_k C_h(e^j_{-3},x)- C_k^i(e^j_{-3},x) -2\delta^i_j C_h(e^k_{-3},x)+C_j^i(e^k_{-3},x)-\\
  &C^i_{-2}(e^k_{-3},e^j_{-3})=0.
  \end{align*}
  Solving these equations we get that:
  \begin{align}
  C^i_{-2}(e^k_{-3},e^j_{-3})&=D_{e^j_{-1}}(W_3)^i_k-D_{e^k_{-1}}(W_3)^i_j, \\
  C_k^i(e^j_{-3},x)&=D_{e^k_{-1}}(W_3)^i_j - D_{e^j_{-1}}(W_3)^i_k-D_{e^k_{-2}}(W_3)^i_j,\\
\label{l09} C_h(e^j_{-3},x)&= -\frac{1}{2}\tr^i_j\left( D_{e^k_{-1}}(W_3)^i_j - 2D_{e^j_{-1}}(W_3)^i_k+D_{e^j_{-2}}(W_2)^i_k\right),\\
  C^i_{-1}(e^k_{-2},e^j_{-3})&=\tr_0\left( D_{e^k_{-1}}(W_3)^i_j - 2D_{e^j_{-1}}(W_3)^i_k+D_{e^j_{-2}}(W_2)^i_k\right)=0,
  \end{align}
 where $\tr_0$ is the trace-free part of the tensor with respect to the indexes $i,j.$ All other parts of the curvature function of the degree 4 are equal to 0.
 
 In the degree 5 there are only 4 terms of the curvature function which should be 0 and can give us conditions on Wilczynski invariants. These terms are $C_j^i(e^k_{-2},e^l_{-3}) $, $C_h(e^k_{-2},e^l_{-3})$, $C_y(e^k_{-1},e^l_{-3})$ and $C_y(e^k_{-2},e^l_{-2})$. The Bianchi identity gives us the following equality: 
 \begin{equation}\label{l06}
 \d C_h(e^l_{-1},e^j_{-3},x)=-C_h(e^l_{-2},e^j_{-3})+C_y(e^l_{-1},e^j_{-3})=D_{e^l_{-1}}C_h(e^j_{-3},x).
 \end{equation}
 From \eqref{l06} and \eqref{l09} we obtain that the term
 \[
 D_{e^l_{-1}}\left( D_{e^k_{-1}}(W_3)^i_j - 2D_{e^j_{-1}}(W_3)^i_k+D_{e^j_{-2}}(W_2)^i_k\right)
 \]
 should be equal to 0 for conformal geodesics equation.
 After computations of the degree 4 we know that 
 \[C(e^l_{-1},e^j_{-2})=C_y(e^l_{-1},e^j_{-2})=I_4\cdot y.\] 
 Now we apply formula \eqref{l07} to the invariant $W_2$:
 \[ D_{e^l_{-1}} D_{e^k_{-2}}W_2-D_{e^l_{-2}} D_{e^k_{-1}}W_2=(I_4)_{lk}\cdot y.W_2=0.\]
 Therefore $D_{e^l_{-1}} D_{e^k_{-2}}W_2=D_{e^k_{-2}} D_{e^l_{-1}}W_2=0$ and
 \begin{equation} \label{l08} D_{e^l_{-1}}\left( D_{e^k_{-1}}(W_3)^i_j - 2D_{e^j_{-1}}(W_3)^i_k\right)=0
 \end{equation}
 Finally, we can easily deduce from \eqref{l08} that $ D_{e^l_{-1}} D_{e^k_{-1}}W_3=0.$ This leads us to the fact that 
 \[ D_{e^l_{-1}}C^i_{-2}(e^k_{-3},e^j_{-3})=D_{e^l_{-1}}C_k^i(e^j_{-3},x)=0.\]
From equality $D_{e^k_{-1}}C_h(e^j_{-3},x)=0$ follows that all terms $C_j^i(e^k_{-2},e^l_{-3}) ,$ $C_h(e^k_{-2},e^l_{-3}),$ $C_y(e^k_{-1},e^l_{-3}),$ and $C_y(e^k_{-2},e^l_{-2})$ are equal to 0.
 
 In the degree 6 the only term $C_y(e^k_{-2},e^l_{-3})$ gives us conditions on invariant. From the Bianchi identity we get that the following expression should be equal to 0:
 \[\d C_y(e^j_{-1},e^k_{-3},x)= -C_y(e^j_{-2},e^k_{-3})=D_{e^j_{-1}}C_y(e^k_{-3},x).\]
 The coefficient $C_y(e^k_{-3},x)$ could be found from the following equality:
 \[ D_{e^k_{-2}}(W_3)^i_j=\d C_{-1}^i(e^k_{-2},e^j_{-3},x)=-C_{-1}^i(e^k_{-3},e^j_{-3})-2\delta^i_k C_y(e^j_{-3},x) .\]
 Symmetrization of the previous equality by $i,j$ together with the trace operator gives us formula for $C_y(e^k_{-3},x)$:
  \[ C_y(e^k_{-3},x)=-\frac{1}{2(n+1)} \left(D_{e^i_{-2}}(W_3)^i_j +D_{e^j_{-2}}(W_3)^i_i\right).\]
  
  In the degree 7 there is no conditions on the curvature.
 \end{proof}
 To obtain explicit formulas on invariants we should use formula \eqref{l02}. Covariant derivatives of Wilczynski invariant have the following expressions:
 \begin{align*}
 D_{e^k_{-1}}W_2=&\frac{\d W_2}{\d q^k},\\
 D_{e^k_{-1}}W_3=&\frac{\d W_3}{\d q^k}+2H^{-1}_kW_2, \\
 D_{e^k_{-2}}W_2=&\frac{\d W_2}{\d p^k}-4H^{-1}_kW_2, \\
 D_{e^k_{-2}}W_3=&\frac{\d W_3}{\d p^k}-\frac{\d W_3}{\d q^l}B^l_k-6H^{-1}_kW_3+\\
 &2(H^{-2}_k-H^{-1}_lB^l_k)W_2+G^{*,-2}_{*k}.W_3-W_3.G^{*,-2}_{*k} \\
 \end{align*}
 Using previous formulas we get that conditions of the proposition \ref{p4} are expressed in the following way:
  \begin{align*}
 &\frac{\d W_2}{\d q^k}=0,\\
 &\tr_0\left(\frac{\d (W_2)^i_k}{\d p^j}-2\frac{\d (W_3)^i_k}{\d q^j}+\frac{\d (W_3)^i_j}{\d q^k} -8H^{-1}_j(W_2)^i_k+2H^{-1}_k(W_2)^i_j\right)=0, \\
 &\frac{\d}{\d q^l}\left(-2\frac{\d (W_3)^i_j}{\d q^i}+\frac{\d (W_3)^i_i}{\d q^j}\right)=0, \\
 & \frac{\d}{\d q^l}\left(D_{e^i_{-2}}(W_3)^i_j -D_{e^j_{-2}}(W_3)^i_i\right)+\left(I_4\right)_{ji}\left(W_2\right)^i_l-\\&\qquad\quad H^{-1}_l\left(\frac{\d (W_2)^i_j}{\d p^i}-2\frac{\d (W_3)^i_j}{\d q^i}
 +\frac{\d (W_3)^i_i}{\d q^j} -8H^{-1}_i(W_2)^i_j\right)=0;
 \end{align*}
 The formulas above end our computations.

 \bibliography{my}{}

\begin{thebibliography}{10}

\bibitem{EastP}
R.~Baston and M.~Eastwood.
\newblock {\em The Penrose transform. Its interaction with representation
  theory}.
\newblock The Clarendon Press, Oxford University Press, New York, 1989.

\bibitem{CartanProj}
E.~Cartan.
\newblock Sur les vari\'{e}t\'{e}s \`{a} connection projective.
\newblock {\em Bull. Soc. Math. France}, 52:205--241, 1924.

\bibitem{dkm}
B.~Doubrov, B.~Komrakov, and T.~Morimoto.
\newblock Equivalence of holonomic differential equations.
\newblock {\em Lobachevskij Journal of Mathematics}, 3:39--71, 1999.

\bibitem{medv1}
A.~Medvedev.
\newblock Geometry of third order ode systems.
\newblock {\em Archivium Mathematicum}, 46:351--361, 2010.

\bibitem{medv2}
A.~Medvedev.
\newblock Third order odes systems and its characteristic connections.
\newblock {\em Symmetry, Integrability and Geometry: Methods and Applications
  (SIGMA)}, 7, 2011.

\bibitem{mor93}
T.~Morimoto.
\newblock Geometric structures on filtered manifolds.
\newblock {\em Hokkaido Math.\ J.}, 22:263--347, 1993.

\bibitem{bryant}
M.~Eastwood R.~Bryant, M.~Dunajski.
\newblock Metrisability of two-dimensional projective structures.
\newblock {\em arXiv:0801.0300}, 2008.

\bibitem{Schouten}
J.~A. Schouten.
\newblock {\em Ricci-Calculus. An introduction to tensor analysis and its
  geometrical applications. 2d ed.}
\newblock Berlin, Springer, 1954.

\bibitem{Eastwood}
M.~Eastwood T.~Bailey.
\newblock Conformal circles and parametrizations of curves in conformal
  manifolds.
\newblock {\em Proceedings of the American Mathematical Society}, 108:215--221,
  1990.

\bibitem{tan70}
N.~Tanaka.
\newblock On differential systems, graded lie algebras and pseudo-groups.
\newblock {\em J.\ Math.\ Kyoto.\ Univ.}, 10:1--82, 1970.

\bibitem{CapCor}
A.~\v{C}ap.
\newblock Correspondence spaces and twistor spaces for parabolic geometries.
\newblock {\em J. Reine Angew. Math.}, 582:143--172, 2005.

\bibitem{CapTwo}
A.~\v{C}ap.
\newblock Two constructions with parabolic geometries.
\newblock In {\em Proceedings of the 25th Winter School on Geometry and
  Physics, Srni 2005}, volume~79 of {\em Rend. Circ. Mat. Palermo Suppl. ser.
  II}, pages 11--37, 2006.

\bibitem{CapSloBook}
A.~\v{C}ap and J.~Slov\'{a}k.
\newblock {\em Parabolic geometries. I. Background and general theory.}
\newblock Mathematical Surveys and Monographs, 154. American Mathematical
  Society, Providence, 2009.

\bibitem{yano}
K.~Yano.
\newblock {\em The Theory Of Lie derivatives and its applications}.
\newblock North Holland Publishing Co. - Amsterdam, 1955.

\end{thebibliography}
\bibliographystyle{plain}

\end{document}